\documentclass[12pt]{amsart}
\usepackage{amssymb,amsmath,amsfonts,latexsym}
\usepackage{bm}

\usepackage{array,graphics,color}
\numberwithin{equation}{section}

\newtheorem{propn}{Proposition}[section]
\newtheorem{thm}[propn]{Theorem}
\newtheorem{lemma}[propn]{Lemma}

\newtheorem*{thm*}{Theorem}
\theoremstyle{definition}
\newtheorem{defn}[propn]{Definition}

 \newcommand{\D}{\mathbb{D}}

\newcommand{\clb}{\mathcal{B}}

\newcommand{\cle}{\mathcal{E}}
\newcommand{\clf}{\mathcal{F}}

\newcommand{\clq}{\mathcal{Q}}

\newcommand{\z}{\bm{z}}
\newcommand{\w}{\bm{w}}
\newcommand{\raro}{\rightarrow}
\newcommand{\NI}{\noindent}

\begin{document}
\title{ Toeplitz and Asymptotic Toeplitz operators on $H^2(\mathbb{D}^n)$ }

\author[Maji]{Amit Maji}
\address{Indian Statistical Institute, Statistics and Mathematics Unit, 8th Mile, Mysore Road, Bangalore, 560059, India}
\email{amaji\_pd@isibang.ac.in, amit.iitm07@gmail.com}

\author[Sarkar]{Jaydeb Sarkar}
\address{Indian Statistical Institute, Statistics and Mathematics Unit, 8th Mile, Mysore Road, Bangalore, 560059, India}
\email{jay@isibang.ac.in, jaydeb@gmail.com}

\author[Sarkar]{Srijan Sarkar}
\address{Indian Statistical Institute, Statistics and Mathematics Unit, 8th Mile, Mysore Road, Bangalore, 560059, India}
\email{srijan\_rs@isibang.ac.in, srijansarkar@gmail.com}


\begin{abstract}

We initiate a study of asymptotic Toeplitz operators on the Hardy
space $H^2(\mathbb{D}^n)$ (over the unit polydisc $\mathbb{D}^n$ in
$\mathbb{C}^n$). We also study the Toeplitz operators in the
polydisc setting. Our main results on Toeplitz and  asymptotic
Toeplitz operators can be stated as follows: Let $T_{z_i}$ denote
the multiplication operator on $H^2(\mathbb{D}^n)$ by the $i^{th}$
coordinate function $z_i$, $i =1, \ldots, n$, and let $T$ be a
bounded linear operator on $H^2(\mathbb{D}^n)$. Then the following
hold:

(i) $T$ is a Toeplitz operator (that is, $T = P_{H^2(\mathbb{D}^n)}
M_{\varphi}|_{H^2(\mathbb{D}^n)}$, where $M_{\varphi}$ is the
Laurent operator on $L^{2}(\mathbb{T}^n)$ for some $\varphi \in
L^\infty(\mathbb{T}^n)$) if and only if $T_{z_i}^* T T_{z_i} = T$
for all $i = 1, \ldots, n$.

(ii) $T$ is an asymptotic Toeplitz operator if and only if $T =
\mbox{~Toeplitz} + \mbox{~compact}$.

\noindent The case $n = 1$ is the well known results of Brown and
Halmos, and Feintuch, respectively. We also present related results
in the setting of vector-valued Hardy spaces over the unit disc.
\end{abstract}


\maketitle

\section{Introduction}

Although concrete bounded linear operators on Hilbert spaces exist
in great variety and can exhibit interesting properties, one of the
main concerns of function theory and operator theory has generally
been the study of operators which are connected with the spaces of
holomorphic and integrable functions. The class of Toeplitz and
analytic Toeplitz operators have turned out to be one of the most
important classes of concrete operators from this point of view.

Toeplitz operators on the Hardy space (or, on the $l^2$ space) were
first studied by O. Toeplitz (and then by P. Hartman and A. Wintner
in \cite{Hart}). However, a systematic study of Toeplitz operators
was triggered by the seminal paper of Brown and Halmos \cite{BROW}
on algebraic properties of Toeplitz operators on $H^2(\mathbb{D})$.
Here $H^2(\mathbb{D})$ denote the Hardy space over the open unit
disc $\mathbb{D}$ in $\mathbb{C}$. The study of Toeplitz operators
on Hilbert spaces of holomorphic functions, like the Hardy space,
the Bergman space and the weighted Bergman spaces, on domains in
$\mathbb{C}^n$ is also one of the very active area of current
research that brings together several areas of mathematics. For more
information on this direction of research, we refer the reader to
\cite{Bot}, \cite{Choe}, \cite{DES}, \cite{XD}, \cite{Ron},
\cite{EK},  \cite{SD}, \cite{Up} and the references therein.

Recall that the well-known Brown-Halmos theorem characterizes
Toeplitz operators on $H^2(\mathbb{D})$ as follows (see the
matricial characterization, Theorem 6 in \cite{BROW}): Let $T$ be a
bounded linear operator on $H^2(\mathbb{D})$. Then $T$ is a Toeplitz
operator if and only if
\[
T_z^{*} T T_z = T.
\]

One of the main results of this paper is the following
generalization of Brown-Halmos theorem (see Theorem
\ref{thm-ToeplitzD}): A bounded linear operator $T$ on
$H^2(\mathbb{D}^n)$ is a Toeplitz operator if and only if
\[
T_{z_j}^* T T_{z_j} = T,
\]
for all $j = 1, \ldots, n$ (see Section 2 for notation and
background definitions).

The notion of Toeplitzness was extended to more general settings by
Barr\'{\i}a and Halmos \cite{BARR} and Feintuch \cite{FEN}. Also see
Popescu \cite{GP} for Toeplitzness in the non-commutative setting.
Accordingly, following Feintuch (and Barr\'{\i}a and Halmos
\cite{BARR}) we shall say that a bounded linear operator $T$ on
$H^2(\mathbb{D})$ is (uniformly) \textit{asymptotically Toeplitz} if
$\{T_z^{*m}T T_z^m \}_{m \geq 1}$ converges in operator norm. The
following theorem due to Feintuch \cite{FEN} gives a remarkable
characterization of asymptotically Toeplitz operators: A bounded
linear operator $T$ on $H^2(\mathbb{D})$ is asymptotically Toeplitz
if and only if $T = \mbox{~Toeplitz} + \mbox{~compact}$.

After a preliminary section (Section 2) on the Hardy space over unit
polydisc, in Section 3, we introduce the asymptotic Toeplitz
operators in polydisc setting (see Definition \ref{def:ATO}). In
Theorem \ref{th2}, we prove the following generalization of
Feintuch's theorem: A bounded linear operator $T$ on
$H^2(\mathbb{D}^n)$ is asymptotically Toeplitz  if and only if $T =
\mbox{~Toeplitz} + \mbox{~compact}$.

In Section 4, we investigate Toeplitzness and asymptotic
Toeplitzness of compressions of the $n$-tuple of multiplication
operators $(T_{z_1}, \ldots, T_{z_n})$ to Beurling type quotient
spaces of $H^2(\D^n)$. Note that a rich source of $n$-tuples of
commuting contractions comes from quotient Hilbert spaces of
$H^2(\D^n)$. More specifically, let $\theta \in H^\infty(\D^n)$ be
an inner function, that is, $|\theta| = 1$ on the distinguished
boundary $\mathbb{T}^n$ of $\D^n$. Set
\[
\mathcal{Q}_{\theta} = H^2(\mathbb{D}^n) \ominus \theta
H^2(\mathbb{D}^2),
\]
and
\[
C_{z_i} = P_{\mathcal{Q}_{\theta}} T_{z_i}|_{\mathcal{Q}_{\theta}},
\]
where $P_{\clq_{\theta}}$ denotes the orthogonal projection from
$H^2(\mathbb{D}^n)$ onto $\mathcal{Q}_{\theta}$. A basic question is
now to characterize those $T \in \mathcal{B}(\mathcal{Q}_{\theta})$
for which
\[
C_{z_i}^* T C_{z_i} = T.
\]
Similarly, characterize those $T \in
\mathcal{B}(\mathcal{Q}_{\theta})$ for which
\[
C_{z_i}^{*m} T C_{z_i}^m \raro A,
\]
in norm, for some $A \in \mathcal{B}(\clq_{\theta})$ and for all $i
= 1, \ldots, n$. In this general setting, to remedy the subtlety of
the product domain $\D^n$, we modify the above condition by adding
another natural condition. The main content of Section 4 is the
following: Let $T, A \in \clb(\clq_{\theta})$. Then $C_{z_i}^* A
C_{z_i} = A$ for all $i = 1, \ldots, n$, if and only if $A = 0$.
Moreover, the following are equivalent:

(i) $C_{z_i}^{*m} T C_{z_i}^m \raro A$ and $C_{z_i}^{*m} (T -A)
C_{z_j}^m \raro 0$ in norm for all $i, j = 1, \ldots, n$;

(ii) $C_{z_i}^{*m} T C_{z_i}^m \raro 0$ in norm for all $i = 1,
\ldots, n$;

(iii) $T$ is compact.

In Section 5, we study the above questions in the vector-valued
Hardy space over the unit disc setting. To be precise, let
$\mathcal{E}$ be a Hilbert space, and let $\Theta \in
H^{\infty}_{\mathcal{B}(\mathcal{E})}(\mathbb{D})$ be an inner
multiplier \cite{NAGY}. Then the \textit{model space} and the
\textit{model operator} are defined by $\clq_{\Theta} =
H^2_{\mathcal{E}}(\mathbb{D}) \ominus \Theta
H_{\mathcal{E}}^2(\mathbb{D})$ and $S_{\Theta} = P_{\clq_{\Theta}}
T_z|_{\clq_{\Theta}}$, respectively. We prove that for every $T \in
\mathcal{B}(\clq_{\Theta})$, the following holds: (i)
$S_{\Theta}^{*}TS_{\Theta}= T$ if and only if $T = 0$, and (ii)
$\{S_{\Theta}^{*m} T S_{\Theta}^m\}_{m\geq 1}$ converges in norm if
and only if $T$ is compact.

\section{Preliminaries}

Let $n \geq 1$ and $\mathbb{D}^n$ be the open unit polydisc in
$\mathbb{C}^n$. In the sequel, $\z$ will always denote a vector
$\z=\left( z_1, \ldots, z_n\right)$ in $\mathbb{C}^n$.

The \textit{Hardy space} $H^2(\mathbb{D}^n)$ over $\mathbb{D}^n$ is
the Hilbert space of all holomorphic functions $f$ on $\mathbb{D}^n$
such that
\[
\|f\|_{H^2(\mathbb{D}^n)}:= \left(\sup_{0\leq r< 1}
\int_{\mathbb{T}^n}|f(r e^{i \theta_1}, \ldots, r e^{i
\theta_n})|^2~d \bm{\theta} \right)^{\frac{1}{2}}< \infty,
\]
where $d\bm{\theta}$ is the normalized Lebesgue measure on the torus
$\mathbb{T}^n$, the distinguished boundary of $\mathbb{D}^n$. Let
$(T_{z_1}, \ldots, T_{z_n})$ denote the $n$-tuple of multiplication
operators by the coordinate functions $\{z_i\}_{i=1}^n$, that is,
\[
(T_{z_i} f)(\w) = w_i f(\w),
\]
for all $\w \in \D^n$ and $i = 1, \ldots, n$. We will often identify
$H^2(\D^n)$ with the $n$-fold Hilbert space tensor product of one
variable Hardy space as $H^2(\D) \otimes \cdots \otimes H^2(\D)$. In
this identification, $T_{z_i}$ can be represented as
\[
I_{H^2(\D)} \otimes \cdots \otimes \underbrace{T_z}\limits_{i^{th}~
place} \otimes \cdots \otimes I_{H^2(\D)},
\]
for all $i = 1, \ldots, n$. Also one can identify the Hardy space
(via the radial limits of functions in $H^2(\mathbb{D}^n)$) with the
closed subspace of $L^2(\mathbb{T}^n)$ in the following sense: Let
$\left \{ e_{\bm k} : \bm{k} \in \mathbb{Z}^{n} \right \}$ be the
orthonormal basis of $L^{2}({\mathbb{T}^n})$, where $\bm{k} = (k_1,
\ldots, k_n) \in \mathbb{Z}^{n}$ and $e_{\bm k} = e^{i \theta_1 k_1}
\cdots e^{i \theta_n k_n}$. Then a function
\[
\displaystyle\sum_{\bm{k} \in \mathbb{Z}^n} a_{\bm{k}} e_{\bm{k}}
\in L^2(\mathbb{T}^n),
\]
is the radial limit function of some function in $H^2(\mathbb{D}^n)$
if and only if $a_{\bm{k}} = 0$ whenever at least one of the $k_j$,
$j = 1, \ldots, n$, is negative. In particular, the set of all
monomials $\{\bm{z}^{\bm{k}} : \bm{k} \in \mathbb{Z}^n_+\}$ form an
orthonormal basis for $H^2(\mathbb{D}^n)$, where $\bm{k} = (k_1,
\ldots, k_n) \in \mathbb{Z}^n_+$ and $\bm{z}^{\bm{k}} = z_1^{k_1}
\cdots z_n^{k_n}$ (cf. \cite{Ahern}, \cite{Rudin}). We use
$P_{H^2(\mathbb{D}^n)}$ to denote the orthogonal projection from
$L^2(\mathbb{T}^n)$ onto $H^2(\mathbb{D}^n)$, that is,
\[
P_{H^2(\mathbb{D}^n)} (\sum_{\bm{k} \in \mathbb{Z}^n} a_{\bm{k}}
e_{\bm{k}}) = \sum_{\bm{k} \in \mathbb{Z}^n_+} a_{\bm{k}}
e_{\bm{k}},
\]
for all $\displaystyle\sum_{\bm{k} \in \mathbb{Z}^n} a_{\bm{k}}
e_{\bm{k}}$ in $L^2(\mathbb{T}^n)$.

For $\varphi \in L^{\infty}(\mathbb{T}^n)$, the \textit{Toeplitz
operator} with symbol $\varphi$ is the operator $T_{\varphi} \in
\mathcal{B}(H^2(\mathbb{D}^n))$ defined by
\[
T_{\varphi} f = P_{H^2(\mathbb{D}^n)} (M_{\varphi} f) \quad \quad (f
\in H^2(\mathbb{D}^n)),
\]
where $M_{\varphi}$ is the Laurent operator on $L^2(\mathbb{T}^n)$
defined by $M_{\varphi} g = \varphi g$ for all $g \in
L^2(\mathbb{T}^n)$. Therefore
\[
T_{\varphi} = P_{H^2(\mathbb{D}^n)}
M_{\varphi}|_{H^2(\mathbb{D}^n)}.
\]
For the relevant results on Toeplitz operators on
$H^2(\mathbb{D}^n)$ we refer the reader to \cite{Bot, Choe, XD, SD,
 NS} and references therein.

The following lemma will prove useful in what follows.

\begin{lemma}\label{sec1.lem1}
Let $\mathcal{H}$ be a Hilbert space and $A \in
\mathcal{B}(\mathcal{H})$ be a compact operator. If $R$ is a
contraction on $\mathcal{H}$, and if $R^{*m} \rightarrow 0$ in
strong operator topology, then $R^{*m}A \rightarrow 0$ in norm.
\end{lemma}
\begin{proof}
This is a particular case of (\cite{Bot}, 1.3 (d), page 3).
\end{proof}

In what follows, for each $\bm{k} \in \mathbb{Z}^n_+$ and $\bm{l}
\in \mathbb{Z}^n$, we write $T_{\bm z}^{\bm{k}} = T_{z_1}^{k_1}
\cdots T_{z_n}^{k_n}$, $ M_{e^{i \theta}}^{{{\bm{l}}}} =
M_{e^{i\theta_1}}^{l_1} \cdots M_{e^{i\theta_n}}^{l_n}$, $T_{\bm
z}^{* \bm{k}} = T_{z_1}^{*k_1} \cdots T_{z_n}^{*k_n}$ and $ M_{e^{i
\theta}}^{*{{\bm{l}}}} = M_{e^{i\theta_1}}^{*l_1} \cdots
M_{e^{i\theta_n}}^{*l_n}$.

\section{Toeplitz operators in several variables}

In the following we prove a generalization of Brown and Halmos
characterization \cite{BROW} of Toeplitz operators on
$H^2(\mathbb{D})$. This result should be compared with the algebraic
characterization of Guo and Wang \cite{Guo} which states that $T$ in
$\mathcal{B}(H^2(\mathbb{D}^n))$ is a Toeplitz operator if and only
if $T_{\varphi}^* T T_{\varphi} = T$ for all inner function $\varphi
\in H^\infty(\mathbb{D}^n)$.

\begin{thm}\label{thm-ToeplitzD}
Let $T \in \mathcal B\left(H^{2}(\mathbb{D}^n) \right)$. Then $T$ is
a Toeplitz operator if and only if $T_{z_j}^{*}T T_{z_j} = T$ for
all $j = 1, \ldots n$.
\end{thm}

\begin{proof}
For each $k\in \mathbb{Z}_+$, define ${{\bm{k}}_d} \in \mathbb{Z}^n_+$
by ${{\bm{k}}_d}=(k,\ldots,k)$. From $T_{z_j}^{*}T T_{z_j} = T$, $j
= 1, \ldots n$, we obtain that
\[
T_{\bm z}^{* \bm{k}_d} T T_{\bm z}^{\bm{k}_d} = T,
\]
which implies that
\[
\begin{split}
\langle T e_{\bm {i} +\bm{k}_{d}}, e_{\bm j + \bm k_{d}} \rangle & =
\langle T T_{\bm z}^{\bm{k}_d} e_{\bm {i}}, T_{\bm z}^{\bm{k}_d}
e_{\bm j} \rangle
\\
& = \langle T e_{\bm {i}}, e_{\bm {j}} \rangle,
\end{split}
\]
for all $k \in \mathbb{Z}_+$ and $\bm{i},\bm{j} \in \mathbb{Z}^n_+$.
Now for each $\bm{l}, \bm{m} \in \mathbb{Z}^n$, there exists $ \bm
t= (t_1, \ldots t_n) \in \mathbb{Z}^n_+$ such that $ \bm{l} +
{{\bm{k}}_d}, \bm{m} + {{\bm{k}}_d} \in \mathbb{Z}^n_+$ for all
${{\bm{k}}_d} \geq \bm t$ (that is, $k \geq t_j$ for all $j = 1,
\ldots, n$). Hence setting
\[
A_{k} = M_{e^{i \theta}}^{*{{\bm{k}}_d}} T P_{H^2(\mathbb{D}^n)}
M_{e^{i \theta}}^{{\bm{k}}_d},
\]
for each $k \geq 1$, we have
\[
\begin{split}
\langle A_{k}e_{\bm l}, e_{\bm m} \rangle_{L^2(\mathbb{T}^n)} & =
\langle T P_{H^2(\mathbb{D}^n)} M_{e^{i \theta}}^{{\bm{k}}_d} e_{\bm
l}, M_{e^{i \theta}}^{{\bm{k}}_d} e_{\bm m}
\rangle_{L^2(\mathbb{T}^n)}
\\
& = \langle TP_{H^2(\mathbb{D}^n)} e_{\bm {l}+ {\bm{k}}_d}, e_{\bm
{m}+{\bm{k}}_d} \rangle_{L^2(\mathbb{T}^n)},
\end{split}
\]
and therefore, for all ${{\bm{k}}_d} \geq \bm t$, we have that
\[
\begin{split}
\langle A_{k}e_{\bm l}, e_{\bm m} \rangle_{L^2(\mathbb{T}^n)} & =
\langle T e_{\bm{l}+ {{\bm{k}}_d}}, e_{\bm{m}+{{\bm{k}}_d}}
\rangle_{H^2(\mathbb{D}^n)}
\\
& = \langle T e_{\bm {l+ t }}, e_{\bm {m+t }}
\rangle_{H^2(\mathbb{D}^n)}.
\end{split}
\]
This implies in particular that
\[
\langle A_{k}e_{\bm l}, e_{\bm m} \rangle  \raro \langle T e_{\bm
{l+t}}, e_{\bm {m+t}}\rangle  \text{ as } k \raro \infty.
\]
Let the bilinear form $\eta$ on the linear span of $\left \{ e_{\bm
s} : \bm{s} \in \mathbb{Z}^{n} \right \}$ be defined by
\[
\eta (e_{\bm{l}},e_{\bm{m}})=\lim_{k\rightarrow \infty} \langle
A_{k}e_{\bm l}, e_{\bm m} \rangle,
\]
for all $\bm{l}, \bm{m} \in \mathbb{Z}^n$. Since $\|A_k \| \leq \|T
\|$, $k \geq 1$, it follows that $\eta$ is a bounded bilinear form.
Therefore, $\eta$ can be extended to a bounded bilinear form (again
denoted by $\eta$) on all of $L^2(\mathbb{T}^n)$, and hence there
exists a unique bounded linear operator $A_{\infty}$ on
$L^{2}(\mathbb{T}^n)$ such that
\[
\eta(f,g)=\langle A_{\infty}f,g \rangle = \lim_{k\rightarrow \infty}
\langle A_{k}f, g \rangle,
\]
for all $f, g \in L^2(\mathbb{T}^n)$. Now let $j \in \{1, \ldots,
n\}$, ${\bm l}, {\bm m} \in \mathbb{Z}^n$ and set
\[
\epsilon_j =(0, \ldots, \underbrace{1}_{j^{th} ~place}, \ldots, 0).
\]
Then for all $k$ sufficiently large (depending on ${\bm l}, {\bm m}$
and $j$ ), we have
\[
\begin{split}
\langle (M_{e^{i \theta}}^{*{{\bm{k}}_d}} T P_{H^2(\mathbb{D}^n)}
M_{e^{i \theta}}^{{{\bm{k}}_d}}) e_{{\bm l} + \epsilon_j}, e_{{\bm
m} + \epsilon_j} \rangle_{L^2(\mathbb{T}^n)} & =  \langle T
P_{H^2(\mathbb{D}^n)} e_{{\bm l+{{\bm{k}}_d}}+\epsilon_j}, e_{{\bm
m} +{{\bm{k}}_d}+\epsilon_j} \rangle_{L^2(\mathbb{T}^n)}
\\
& = \langle T e_{{\bm l+{{\bm{k}}_d}}+\epsilon_j}, e_{{\bm m}
+{{\bm{k}}_d}+\epsilon_j} \rangle_{H^2(\mathbb{D}^n)}
\\
& = \langle T_{z_j}^*TT_{z_j} e_{{\bm l+{{\bm{k}}_d}}}, e_{{\bm m}
+{{\bm{k}}_d}} \rangle_{H^2(\mathbb{D}^n)}
\\
& = \langle T e_{{\bm l+{{\bm{k}}_d}}}, e_{{\bm m} +{{\bm{k}}_d}}
\rangle_{H^2(\mathbb{D}^n)}
\\
& = \langle A_{k}e_{\bm l}, e_{\bm m} \rangle_{L^2(\mathbb{T}^n)}.
\end{split}
\]
Therefore
\[
\begin{split}
\langle A_{\infty} e_{{\bm l} + \epsilon_j}, e_{{\bm m} +
\epsilon_j} \rangle_{L^2(\mathbb{T}^n)} & = \lim_{k \rightarrow
\infty} \langle (M_{e^{i \theta}}^{*{{\bm{k}}_d}} T
P_{H^2(\mathbb{D}^n)} M_{e^{i \theta}}^{{{\bm{k}}_d}}) e_{{\bm l} +
\epsilon_j}, e_{{\bm m} + \epsilon_j} \rangle_{L^2(\mathbb{T}^n)}
\\
& = \langle A_{\infty} e_{\bm l}, e_{\bm m}
\rangle_{L^2(\mathbb{T}^n)},
\end{split}
\]
and consequently $M^*_{e^{i \theta_j}} A_{\infty} M_{e^{i \theta_j}}
= A_{\infty}$, that is, $A_{\infty} M_{e^{i \theta_j}} =
M_{e^{i\theta_j}} A_{\infty}$. Hence there exists $\varphi$ in
$L^{\infty}(\mathbb{T}^{n})$ such that $A_{\infty} = M_{\varphi}$
\cite{NAGY}. Finally, we note that for $f, g \in
H^{2}(\mathbb{D}^n)$,
\[
\begin{split}
\langle A_{\infty}f, g \rangle_{L^2(\mathbb{T}^n)} & = \lim_{k
\rightarrow \infty} \langle M_{e^{i \theta}}^{*{\bm k_d}} T
P_{H^2(\mathbb{D}^n)} M_{e^{i \theta}}^{\bm {k}_d} f, g
\rangle_{L^2(\mathbb{T}^n)}
\\
& = \lim_{ k \rightarrow \infty} \langle T_{\bm{z}}^{* {\bm k_d}}T
T_{\bm{z}}^{\bm{k}_d} f, g \rangle_{H^2(\mathbb{D}^n)},
\end{split}
\]
that is,
\[
\langle A_{\infty}f, g \rangle_{L^2(\mathbb{T}^n)} = \langle Tf, g
\rangle_{H^2(\mathbb{D}^n)},
\]
and hence
\[
\begin{split}
\langle P_{H^2(\mathbb{D}^n)} A_{\infty}f, g
\rangle_{H^2(\mathbb{D}^n)} & = \langle A_{\infty}f, g
\rangle_{L^2(\mathbb{T}^n)}
\\
& = \langle Tf, g \rangle_{H^2(\mathbb{D}^n)}.
\end{split}
\]
Therefore, $T = P_{H^2(\mathbb{D}^n)}
A_{\infty}|_{H^{2}(\mathbb{D}^n)}= P_{H^2(\mathbb{D}^n)}
M_{\varphi}|_{H^{2}(\mathbb{D}^n)}$, that is, $T$ is a Toeplitz
operator.

\noindent Conversely, let $\varphi \in L^{\infty}(\mathbb{T}^n)$ and
$T = P_{H^2(\mathbb{D}^n)}M_{\varphi}|_{H^{2}(\mathbb{D}^n)}$. Then
for $f, g \in H^{2}(\mathbb{D}^n)$ and $j = 1, \ldots n$, we have
\[
\begin{split}
\langle (T_{z_j}^* T T_{z_j})f, g \rangle_{H^2(\mathbb{D}^n)} & =
\langle {\varphi} e^{i \theta_j} f, e^{i \theta_j} g
\rangle_{L^2(\mathbb{T}^n)}
\\
& = \langle {\varphi} f, g \rangle_{L^2(\mathbb{T}^n)},
\end{split}
\]
that is,
\[
\langle (T_{z_j}^* T T_{z_j})f, g \rangle_{H^2(\mathbb{D}^n)} =
\langle P_{H^2(\mathbb{D}^n)}M_{\varphi}|_{H^2(\mathbb{D}^n)} f, g
\rangle_{H^2(\mathbb{D}^n)},
\]
and therefore $T_{z_j}^{*}TT_{z_j} = T$ for all $j = 1, \ldots n$,
as desired.
\end{proof}

We now characterize compact operators on $H^2(\mathbb{D}^n)$ in
terms of the multiplication operators $\{T_{z_1}, \ldots,
T_{z_n}\}$. This characterization was proved by Feintuch \cite{FEN}
in the case of $n = 1$.

\begin{thm}\label{th1}
A bounded linear map $T$ on $H^{2}(\mathbb{D}^n)$ is compact if and
only if $T_{z_i}^{*m}TT_{z_j}^{m} \rightarrow 0$ in norm for all
$i,j \in \lbrace1,....,n\rbrace$.
\end{thm}
\begin{proof}
Let $T$ on $H^{2}(\mathbb{D}^n)$ be a bounded operator. First
observe that for each $m \geq 1$, we have
\[
T^m_z T_z^{*m} = I_{H^2(\mathbb{D})} - P_{\mathcal{F}_m},
\]
where
\[
\mathcal{F}_m = \mathbb{C} \oplus z \mathbb{C} \oplus \cdots \oplus
z^{m-1} \mathbb{C},
\]
is an $m$-dimensional subspace of $H^2(\mathbb{D})$. For each $m
\geq 1$, set
\[
F_{m} = \prod_{i=1}^{n}(I_{H^2(\mathbb{D}^n)} -T_{z_i}^{m}
T_{z_i}^{*m}).
\]
Then
\[
\begin{split}
F_{m} & = \prod_{i=1}^{n} (I_{H^2(\mathbb{D})} \otimes \cdots
\otimes \underbrace{{(I_{H^2(\mathbb{D})} - T_{z}^{m}
T_{z}^{*m})}}_{i^{th} place} \otimes \cdots \otimes
I_{H^2(\mathbb{D})})
\\
& = \prod_{i=1}^{n} (I_{H^2(\mathbb{D})} \otimes \cdots \otimes
\underbrace{P_{\mathcal{F}_{m}}}_{i^{th} place} \otimes \cdots
\otimes I_{H^2(\mathbb{D})})
\\
& = P_{\clf_{m}} \otimes \cdots \otimes P_{\clf_{m}},
\end{split}
\]
which gives that $F_{m}$ is a finite rank operator and hence
\[
\tilde{F}_{m} = T F_{m} + F_{m} T - F_{m} T F_{m},
\]
is a finite rank operator, $m \geq 1$. Moreover
\[
\begin{split}
T - \tilde{F}_{m} & = T - (T F_{m} + F_{m} T - F_{m} T F_{m})
\\
& = (I_{H^2(\mathbb{D}^n)} - F_{m}) T (I_{H^2(\mathbb{D}^n)} -
F_{m}).
\end{split}
\]
Finally, observe that
\[
\begin{split}
I_{H^2(\mathbb{D}^n)} - F_{m} & = \sum_{1 \leq i_1 < \cdots < i_l
\leq n}(-1)^{l + 1} T_{z_{i_1}}^{m} \cdots T_{z_{i_l}}^{m}
T_{z_{i_1}}^{* m} \cdots T_{z_{i_l}}^{* m}
\\
& = \sum_{1 \leq i_1 < \cdots < i_l \leq n}(-1)^{l + 1} (T_{z_{i_1}}
\cdots T_{z_{i_l}})^m (T_{z_{i_1}} \cdots T_{z_{i_l}})^{* m},
\end{split}
\]
for all $m \geq 1$. Hence, by hypothesis and the triangle inequality
we have
\[
\|T - \tilde{F}_{m} \| = \|(I_{H^2(\mathbb{D}^n)} - F_{m}) T
(I_{H^2(\mathbb{D}^n)} - F_{m})\| \raro 0,
\]
as $m \raro \infty$, that is, $T$ is a compact operator.

\noindent The converse follows from Lemma \ref{sec1.lem1}. This
completes the proof.
\end{proof}

In view of the preceding theorem, it seems reasonable to define
asymptotic Toeplitz operators as follows (compare this with Feintuch
\cite{FEN} and Barr\'{\i}a and Halmos \cite{BARR}):

\begin{defn}\label{def:ATO}
A bounded linear operator $T$ on $H^2(\mathbb{D}^n)$ is said to be
an asymptotic Toeplitz operator if there exists $A \in
\mathcal{B}(H^2(\mathbb{D}^n))$ such that $T_{z_i}^{*m} T
T_{z_i}^{m} \rightarrow A$ and $T_{z_i}^{*m}(T-A)T_{z_j}^{m}
\rightarrow 0$ as $m \rightarrow \infty$ in norm, $ 1 \leq i,j \leq
n$.
\end{defn}

We close this section by characterizing asymptotic Toeplitz
operators on $H^2(\mathbb{D}^n)$ as analogous characterization of
asymptotic Toeplitz operators on $H^2(\mathbb{D})$ (see \cite{FEN}
and also Theorem \ref{sec1.lem2} in Section 5).

\begin{thm}\label{th2}
Let $T$ be a bounded linear operator on $H^{2}(\mathbb{D}^n)$. Then
$T$ is an asymptotic Toeplitz operator if and only if $T$ is a
compact perturbation of Toeplitz operator.
\end{thm}

\begin{proof}
Let $A \in \mathcal{B}(H^2(\mathbb{D}^n))$, $T_{z_i}^{*m} T
T_{z_i}^{m} \rightarrow A$ and $T_{z_i}^{*m}(T-A)T_{z_j}^{m}
\rightarrow 0$ in norm, as $m \raro \infty$, and $1 \leq i,j \leq
n$. Then for all $m \geq 1$,
\[
\begin{split}
\|A - T_{z_j}^* A T_{z_j}\| & \leq \|A - T_{z_j}^{*(m+1)} T
T_{z_j}^{m+1}\| + \|T_{z_j}^{*(m+1)} T T_{z_j}^{m+1} - T_{z_j}^* A
T_{z_j}\|
\\
& \leq \|A - T_{z_j}^{*(m+1)} T T_{z_j}^{m+1}\| + \|T_{z_j}^{*m} T
T_{z_j}^{m} - A\|,
\end{split}
\]
yields $T_{z_j}^{*} A T_{z_j} = A$ for all $j = 1, \ldots, n$. Also
by Theorem \ref{th1}, $T-A$ is compact on $H^{2}(\mathbb{D}^n)$.

\noindent The converse follows from Lemma \ref{sec1.lem1} and
Theorem \ref{thm-ToeplitzD}. This completes the proof.
\end{proof}

\section{Quotient spaces of $H^2(\mathbb{D}^n)$}

The purpose of this section is to extend some of the results of
Section 3 in the case when the ambient operator is the compression
of $(T_{z_1}, \ldots, T_{z_n})$ to a joint $(T^*_{z_1}, \ldots,
T^*_{z_n})$-invariant subspace.

Let $\mathcal{Q}$ be a joint $(T^*_{z_1}, \ldots,
T^*_{z_n})$-invariant subspace of $H^2(\mathbb{D}^n)$. Set
\[
C_{z_i} = P_{\mathcal{Q}} T_{z_i}|_\mathcal{Q},
\]
for all $i = 1, \ldots, n$. Note that $\mathcal{Q}^\perp$ is a joint
invariant subspace for $(T_{z_1}, \ldots, T_{z_n})$ and so
\[
C_{z_i}^* = T^*_{z_i}|_{\mathcal{Q}} \in \mathcal{B}(\mathcal{Q}).
\]
In the case $n = 1$, $C_{z}$ is called a \textit{Jordan block}
\cite{NAGY}. In the several variables quotient space setting, we
have the following analogue of Theorem \ref{th2}.

\begin{thm}\label{quo.thm1}
Let $T, A \in \mathcal{B}(\mathcal{Q})$. Then $C_{z_i}^{*m} T C_{z_i}^m
\raro A$ and $C_{z_i}^{*m} (T -A) C_{z_j}^m \raro 0$ in norm for all
$i, j = 1, \ldots, n$ if and only if $T = A + K$, where $K \in
\mathcal{B}(\mathcal{Q})$ is a compact operator and $C_{z_i}^{*} A
C_{z_i} = A$ for all $i =1, \ldots n$.
\end{thm}

\begin{proof}
We first note that, as in the proof of Theorem \ref{th2}, the
assumption $C_{z_i}^{*m} T C_{z_i}^m \raro A$ as $m \raro \infty$
implies that
\[
C_{z_i}^{*} A C_{z_i} = A,
\]
for all $i =1, \ldots n$. Now it follows from the definition of
$C_{z_i}$ that
\[
C_{z_i}^{*m} = T^{*m}_{z_i}|_{\mathcal{Q}},
\]
and hence
\[
C_{z_i}^{*m} (T -A) C_{z_j}^m = T_{z_i}^{*m}(T-A)P_{\mathcal{Q}}T_{z_j}^{m}|_{\mathcal{Q}},
\]
for all $i, j = 1, \ldots, n$ and $m \geq 1$. By once again using
the fact that
\[
P_{\mathcal{Q}}T_{z_j}^{m} P_{\mathcal{Q}} = P_{\clq} T^{m}_{z_j},
\]
one sees that
\[
T_{z_i}^{*m}(T-A)P_{\mathcal{Q}}T_{z_j}^{m}  =T_{z_i}^{*m}(T-A)
P_{\mathcal{Q}}T_{z_j}^{m}P_{\mathcal{Q}}.
\]
Hence $C_{z_i}^{*m} (T -A) C_{z_j}^m \raro 0$ in $\clb(\clq)$ if and
only if $T_{z_i}^{*m}(T-A)P_{\mathcal{Q}}T_{z_j}^{m}\raro 0$ in
$\clb(H^2(\D^n))$ as $m \raro \infty$.

\NI Therefore, if $C_{z_i}^{*m} (T -A) C_{z_j}^m \raro 0$ as $m
\raro \infty$ in norm for all $i, j = 1, \ldots, n$, then
$T_{z_i}^{*m}(T-A)P_{\mathcal{Q}}T_{z_j}^{m}\raro 0$ in
$\clb(H^2(\D^n))$ as $m \raro \infty$, and consequently by Theorem
\ref{th1}, $(T-A) P_{\clq}$ is a compact operator on
$H^{2}(\mathbb{D}^n)$. Therefore
\[
(T-A) = (T-A) P_{\clq},
\]
is a compact operator on $\clq$, which proves the necessary part.

Conversely, let $T - A $ be a compact operator on $\mathcal{Q}$ and
$C_{z_i}^{*} A C_{z_i} = A$ for all $i =1, \ldots n$. Since
$C_{z_i}^{*m} \raro 0$ as $m \raro \infty$ in the strong operator
topology, Lemma \ref{sec1.lem1} implies that
\[
C_{z_i}^{*m} (T -A) C_{z_j}^m \raro 0,
\]
as $m \raro \infty$. In particular, for all $i =1, \ldots n$
\[
C_{z_i}^{*m} T C_{z_i}^m \raro C_{z_i}^{*m} A C_{z_i}^m.
\]
But $C_{z_i}^{*} A C_{z_i} = A$, $i =1, \ldots n$, yields us
\[
C_{z_i}^{*m} T C_{z_i}^m \raro A.
\]
This completes the proof.
\end{proof}

Considering the particular case $\clq_{\theta} =
H^2(\mathbb{D}^n)/\theta H^2(\mathbb{D}^n)$, where $\theta \in
H^{\infty}(\mathbb{D}^n)$ is an inner function, we get the following
result.

\begin{thm}\label{thm-SVQT}
Let $\theta \in H^{\infty}(\mathbb{D}^n)$ be an inner function and
$\clq_{\theta} = H^2(\mathbb{D}^n)/\theta H^2(\mathbb{D}^n)$ and $A
\in \mathcal{B}(\clq_{\theta})$. Then $C_{z_i}^{*} A C_{z_i} =A$ for
all $i=1, \ldots n$, if and only if $A=0$.
\end{thm}

\begin{proof}
Let $C_{z_i}^* A C_{z_i} = A$ for all $i = 1, \ldots, n$. Since
\[
\clq_{\theta}^\perp = \theta H^2(\D^n),
\]
is a joint invariant subspace for $(T_{z_1}, \ldots, T_{z_n})$, it
follows that
\[
P_{\clq_{\theta}} T_{z_i}^*|_{\clq_{\theta}} = T_{z_i}^*
P_{\clq_{\theta}},
\]
and hence
\[
\begin{split}
A P_{\clq_{\theta}} & = (C_{z_i}^* A C_{z_i}) P_{\clq_{\theta}}
\\
& = (P_{\clq_{\theta}} T_{z_i}^*|_{\clq_{\theta}} A
P_{\clq_{\theta}} T_{z_i}|_{\clq_{\theta}}) P_{\clq_{\theta}}
\\
& = T_{z_i}^* A P_{\clq_{\theta}} T_{z_i} P_{\clq_{\theta}}
\\
& = T_{z_i}^* A P_{\clq_{\theta}} T_{z_i}
\\
& = T_{z_i}^* (A P_{\clq_{\theta}}) T_{z_i}.
\end{split}
\]
for all $i = 1, \ldots, n$. This and Theorem \ref{thm-ToeplitzD}
implies that $AP_{\clq_{\theta}}$ is a Toeplitz operator.
Consequently, there exists $\psi \in L^\infty(\mathbb{T}^n)$ such that
\[
A P_{\clq_{\theta}} = T_{\psi}.
\]
On the other hand, since $T_{\theta}$ is an analytic Toeplitz
operator, it follows that
\[
AP_{\mathcal{Q}}T_{{\theta}}=0.
\]
Hence, using [Theorem 1, C. Gu \cite{CGU}], we conclude that
\[
\begin{split}
T_{\psi \theta} & = T_{\psi} T_{\theta}
\\
& = A P_{\mathcal{Q}_{\theta}} T_{\theta}
\\
& = 0.
\end{split}
\]
This completes the proof of the theorem.
\end{proof}

Summing up the above two results and Lemma \ref{sec1.lem1}, we have
the following generalization of Theorem 1.2 in \cite{CHA}.

\begin{thm}\label{thm-ntheta}
For an inner function $\theta \in H^\infty(\mathbb{D}^n)$ and
bounded linear operators $T$ and $A$ on $\clq_{\theta} =
H^2(\mathbb{D}^n)/\theta H^2(\mathbb{D}^n)$, the following are
equivalent:

(i) $C_{z_i}^{*m} T C_{z_i}^m \raro A$ and $C_{z_i}^{*m} (T -A)
C_{z_j}^m \raro 0$ in norm for all $i, j = 1, \ldots, n$;

(ii) $C_{z_i}^{*m} T C_{z_i}^m \raro 0$ in norm for all $i = 1,
\ldots, n$;

(iii) $T$ is compact.
\end{thm}

For asymptotic Toeplitzness of composition operators on the Hardy
space of the unit sphere in $\mathbb{C}^n$ we refer the reader to
Nazarov and Shapiro \cite{NS}, and Cuckovic and Le \cite{CLe}.

\section{Asymptotic Toeplitz operators on $H^2_{\mathcal{E}}(\mathbb{D})$}

The main purpose of this section is to characterize the compact
operators on the model space $H^2_{\mathbb{C}^p}(\mathbb{D})/ \Theta
H^2_{\mathbb{C}^p}(\mathbb{D})$, where $\Theta \in
H^\infty_{\clb(\mathbb{C}^p)}(\mathbb{D})$ is an inner function. We
note that this result for $p = 1$ case can be found in \cite{CHA}.
Moreover, our proof seems more shorter and conceptually different
(for instance, compare Theorem \ref{sec1.lem3} with Proposition 2.10
in \cite{CHA}).

We begin with the definition of a Toeplitz operator with
operator-valued symbol.

\begin{defn}
Let $\cle$ be a Hilbert space. A bounded linear operator $T$ on
$H^2_{\mathcal{E}}(\mathbb{D})$ is said to be Toeplitz if there
exists an operator-valued function $\Phi$ in
$L_{\mathcal{B}(\mathcal{E})}^{\infty}(\mathbb{T})$ such that
$T=P_{H^2_{\mathcal{E}}(\mathbb{D})}M_{\Phi}
\mid_{H^2_{\mathcal{E}}(\mathbb{D})}$.
\end{defn}

Here let us observe, before we proceed further, the following
characterization of Toeplitz operators on a vector-valued Hardy
space. The result is probably known to the experts but we were not
able to find a reference in the literature. Since the result follows
adapting similar concepts and techniques used in the proof of
Theorem \ref{thm-ToeplitzD}, we give a sketch of the proof.

\begin{thm}\label{thm-TVC}
Let $\mathcal{E}$ be a Hilbert space and $T \in
\mathcal{B}(H^2_{\mathcal{E}}(\mathbb{D}))$. Then $T $ is a Toeplitz
operator if and only if $T_z^{*} T T_z = T$.
\end{thm}
\begin{proof}
Note first that $\{e_m \eta : m \in \mathbb{Z}, \eta \in
\mathcal{E}\}$ is an orthogonal basis of
$L^2_{\mathcal{E}}(\mathbb{T})$, where $e_m = e^{im\theta}$, $m \in
\mathbb{Z}$. For each $k \geq 1$, set
\[
A_k = M_{e^{i\theta}}^{*k} T P M_{e^{i\theta}}^{k},
\]
where $M_{e^{i\theta}}$ is the bilateral shift on
$L^2_{\mathcal{E}}(\mathbb{T})$ and $P$ is the orthogonal projection
from $L^2_{\mathcal{E}}(\mathbb{T})$ onto
$H^2_{\mathcal{E}}(\mathbb{D})$. If $T_{z}^*T T_z = T$ and $k \in
\mathbb{Z}_+$, then
\[
\langle T e_{i+k} \eta, e_{j+k} \zeta \rangle = \langle T e_i \eta,
e_j \zeta \rangle,
\]
for all $i, j \geq 0$. Then for each $l, m \in \mathbb{Z}$, as in
the proof of Theorem \ref{thm-ToeplitzD}, there exists $t \geq 0$
such that $l + k, m + k \geq 0$ for all $k \geq t$, and so
\[
\langle A_k e_l \eta, e_m \zeta \rangle \raro \langle T e_{l+t}
\eta, e_{m+t} \zeta \rangle,
\]
as $k \raro \infty$. Then
\[
(e_l \eta, e_m \zeta) \mapsto \lim_{k\raro \infty} \langle A_k e_l
\eta, e_m \zeta \rangle,
\]
defines a bounded bilinear form on the span of $\{ e_l \eta : l \in
\mathbb{Z}, \eta \in \cle\}$. Therefore, there exists (again,
following the proof of  Theorem \ref{thm-ToeplitzD}) $A_{\infty} \in
\clb(L^2_{\cle}(\mathbb{T}))$ such that
\[
\langle A_{\infty} f, g \rangle = \lim_{k \raro \infty} \langle A_k
f, g \rangle,
\]
for all $f, g \in L^2_{\cle}(\mathbb{T})$. This yields
\[
A_{\infty} M_{e^{i\theta}} = M_{e^{i\theta}} A_{\infty}.
\]
Hence there exists a $\Phi \in
L_{\mathcal{B}(\mathcal{E})}^{\infty}(\mathbb{T})$
such that
\[
A_{\infty}=M_{\Phi},
\]
and hence
\[
T = P_{H_{\mathcal{E}}^2(\mathbb{D})}M_{\Phi}
\mid_{H_{\mathcal{E}}^2(\mathbb{D})}.
\]
The proof of the converse part proceeds verbatim as that of Theorem
\ref{thm-ToeplitzD}. This completes the proof of the theorem.
\end{proof}

Following Feintuch \cite{FEN} we now define an asymptotic Toeplitz
operator on a vector-valued Hardy space.

\begin{defn}
Let $\cle$ be a Hilbert space. A bounded linear operator $T$ on
$H^2_{\cle}(\mathbb{D})$ is said to be an asymptotic Toeplitz
operator if there exists $A \in \mathcal{B}(H^2_{\cle}(\mathbb{D}))$
such that $T_{z}^{*m} T T_{z}^{m} \rightarrow A$  as $m \rightarrow
\infty$ in norm.
\end{defn}

In the theorem below, we generalize the Feintuch's characterization
\cite{FEN} (see also Theorem F, page 195, \cite{NS}) of asymptotic
Toeplitz operators on Hardy space to asymptotic Toeplitz operators
on Hardy space of finite multiplicity. However, the method of proof
here is adapted from the original proof by Feintuch.

\begin{thm}\label{sec1.lem2}
Let $T, A \in \mathcal{B}(H^2_{\mathbb{C}^p}(\mathbb{D}))$. Then
$T_z^{*m} TT_z^m \raro A$ in norm if and only if $A$ is a Toeplitz
operator and $(T -A)$ is compact.
\end{thm}

\begin{proof}
It follows that
\[
\| T_z^{*(m+1)} T T_z^{m+1} - T_z^{*} A T_z \| \leq \| T_z^{*m} T
T_z^{m} - A \| \rightarrow 0
\]
as $m \rightarrow \infty$. This and the triangle inequality yields
$A = T_z^{*} A T_z$. Now let $R_m =T_z^{m} T_z^{*m}$ and
\[
Q_m = I - R_m.
\]
Further, let $P_{\mathbb{C}^p}$ denote the orthogonal projection of
$H^2_{\mathbb{C}^p}(\mathbb{D})$ onto the space of
($\mathbb{C}^p$-valued) constant functions. Since $T_z T_z^* =
I_{H^2_{\mathbb{C}^p}(\mathbb{D})} - P_{\mathbb{C}^p}$, it follows
that
\[
Q_m = \displaystyle \sum_{k=0}^{m-1}T_z^{k} P_{\mathbb{C}^p}
T_z^{*k} \quad \quad (m \geq 1).
\]
Then $Q_m$, $m \geq 1$, is a finite rank operator, and therefore
\[
F_m = (T - A) Q_m + Q_m (T - A) - Q_m (T - A) Q_m \quad \quad (m
\geq 1),
\]
is also a finite rank operator. Moreover
\[
(T - A) - F_m = R_m (T - A) R_m \quad \quad (m \geq 1),
\]
yields
\[
\|(T - A) - F_m\| = \|R_m (T - A) R_m\| \leq  \|T_z^{*m} T T_z^{m} -
A \| \rightarrow 0,
\]
as $m \rightarrow \infty$. So $T - A$ is compact as desired.

\noindent The converse follows from Lemma \ref{sec1.lem1}. This
completes the proof.
\end{proof}

We have the following result in the model space setting.

\begin{propn}\label{sec1.lem3}
Let $\Theta \in H^{\infty}_{\mathcal{B}(\mathcal{E})}(\mathbb{D})$
be an inner multiplier and $T \in \mathcal{B}(\clq_{\Theta})$.
Assume that $\Theta(e^{i \theta})$ is invertible a.e. Then
$S_{\Theta}^{*} T S_{\Theta} =T$ if and only if $T = 0$.
\end{propn}

\begin{proof}

The proof goes exactly along the same lines as the proof of Theorem
\ref{thm-SVQT}. Since
\[
T P_{\clq_{\Theta}} = T_z^{*} (T P_{\clq_{\Theta}}) T_z,
\]
it follows from Theorem \ref{thm-TVC} that $TP_{\mathcal{Q}}$ is a
Toeplitz operator. Consequently, there exists $\Psi \in
L^{\infty}_{\mathcal{B}({\mathcal{E}})}({\mathbb{T}})$ \cite{NAGY}
such that
\[
T P_{\clq_{\Theta}}= T_{\Psi}.
\]
Since $T_{\Theta}$ is an analytic Toeplitz operator, again as in the
proof of Theorem \ref{thm-SVQT}, it follows that
\[
T_{\Psi \Theta} = 0,
\]
and hence
\[
\Psi \Theta =0.
\]
Since $\Theta $ is invertible a.e., it follows that $\Psi =0$ a.e.
and hence $T =0$. This completes the proof.
\end{proof}

Not only is this proposition a considerable generalization of
Proposition 2.10 of \cite{CHA}, but our proof is much simpler. The
principal tool is the identity $S_{\Theta}^{*} =
T_z^{*}|_{\clq_{\Theta}}$.

We have the following characterization which generalizes the
characterization of compact operators on $\clq_{\Theta}$ for $p = 1$
(see the implication (i) and (iii) in Theorem 1.2 in \cite{CHA}).

\begin{thm}\label{thm:compact-invert}
Let $\Theta \in
H^{\infty}_{\mathcal{B}{(\mathbb{C}^p})}(\mathbb{D})$ be an inner
multiplier and $T \in \mathcal{B}(\clq_{\Theta})$. Then $T$ is
compact if and only if $\left \{S_{\Theta}^{*m}TS_{\Theta}^{m}
\right\}_{m \geq 1}$ converges in norm.
\end{thm}

\begin{proof}
If $T$ is compact on $\clq_{\Theta}$, then by Lemma \ref{sec1.lem1},
$\| S_{\Theta}^{*m}TS_{\Theta}^{m} \| \rightarrow 0$ as $m
\rightarrow \infty$. To prove the converse, let $A \in
\mathcal{B}(\clq_{\Theta})$ and $S_{\Theta}^{*m} T S_{\Theta}^{m}
\raro A$, as $m \raro \infty$, in norm. Then by the same argument
used in the proof of Theorem \ref{quo.thm1}, we have
$S_{\Theta}^{*}AS_{\Theta}=A$.  It now follows  from
Proposition \ref{sec1.lem3} that $A=0$ and therefore $T_z^{*m}T P_{\clq_{\Theta}} T_z^m
\raro 0$ as $m \raro \infty$. Now Theorem
\ref{sec1.lem2} implies that $T P_{\clq_{\Theta}}$ is a compact
operator on $H^{2}_{\mathbb{C}^p }(\mathbb{D})$. Therefore $T = T
P_{\clq_{\Theta}}$ is a compact operator on $\clq_{\Theta}$. This
completes the proof.
\end{proof}

Theorem \ref{thm:compact-invert} and Lemma \ref{sec1.lem1} give us
the following generalization of Theorem 1.2 in \cite{CHA}.

\begin{thm}
Let $\Theta \in H^\infty_{\mathcal{B}(\mathbb{C}^p)}(\mathbb{D})$ be
an inner multiplier and $T \in \mathcal{B}(\clq_{\Theta})$. Then the following are
equivalent:

(i) $\{S_{\Theta}^{*m} T S_{\Theta}^m\}_{m\geq1}$ converges in norm;

(ii) $S_{\Theta}^{*m} T S_{\Theta}^m \raro 0$ in norm;

(iii) $T$ is a compact operator.
\end{thm}

\vspace{0.2in}

\noindent\textit{Acknowledgement:}  We are grateful to Professor A.
Bottcher for pointing out that our Lemmaa \ref{sec1.lem1} is a
special case of 1.3 (d) in the monograph by Bottcher and Silbermann
[3]. The first author's research work is supported by NBHM Post
Doctoral Fellowship No. 2/40(50)/2015/ R \& D - II/11569. The second
author is supported in part by NBHM (National Board of Higher
Mathematics, India) grant NBHM/R.P.64/2014.

\end{document}